\documentclass [11pt]{article}
\usepackage{amssymb,amsmath,comment,amsthm}

\def\N{\mathbb{N}}

\def\F{\mathbb{F}}

\def\dis{\displaystyle}

\newtheorem{theorem}{Theorem}[section]

\newtheorem{corollary}[theorem]{Corollary}
\newtheorem{lemma}[theorem]{Lemma}
\newtheorem{definition}[theorem]{Definition}
\newtheorem{remark}[theorem]{Remark}

\newtheorem{examples}[theorem]{Examples}

\begin{document}
\title{Some multiplicative functions over $\F_2$
}
\author{Luis H. Gallardo, Olivier Rahavandrainy \\
Univ Brest, UMR CNRS 6205\\
Laboratoire de Math\'ematiques de Bretagne Atlantique\\
e-mail: luis.gallardo@univ-brest.fr, olivier.rahavandrainy@univ-brest.fr}
\maketitle
\begin{itemize}
\item[] Mathematics Subject Classification (2010): 11T55, 11T06.
\end{itemize}
{\bf{Abstract}}\\
We adapt (over $\F_2$) the general notions of multiplicative function, Dirichlet convolution and Inverse.
We get some interesting results, namely necessary conditions for an odd binary polynomial to be perfect. Note that we are inspired by the ``analogous'' works in \cite{Gall-Rahav-newcongr} and \cite{Touchard}, about odd perfect numbers.

\section{Perfect polynomial over $\F_2$}
Let $A \in \F_2[x]$ be a nonzero polynomial. We say that $A$ is \emph{even} if it has a linear factor and it is
\emph{odd} otherwise. We define a
\emph{Mersenne prime polynomial} over $\F_2$ as an irreducible polynomial of the form
$1+x^a(x+1)^b$, for some positive integers $a,b$.
We say that a divisor $d$ of $A$ is \emph{unitary} if $\displaystyle{\gcd(d,\frac{A}{d}) = 1}$.
Let $\omega(A)$ denote the number of distinct irreducible (or
\emph{prime}) factors of $A$ over $\F_2$ and let $\sigma(A)$ (resp. $\sigma^*(A)$) denote the sum of all divisors (resp. of all unitary divisors) of
$A$ ($\sigma$ and $\sigma^*$ are multiplicative functions). If $\sigma(A) = A$ (resp. $\sigma^*(A) = A$), then we
say that $A$ is \emph{perfect} (resp. \emph{unitary perfect}). The notion of perfect polynomials is introduced (\cite{Canaday})
by E. F. Canaday in $1941$. Many extended studies (\cite{Beard2}, \cite{Gall-Rahav4}, \cite{Gall-Rahav7},
\cite{Gall-Rahav5},
\cite{Gall-Rahav-Mers}) allow to give a list of such polynomials. We get:\\
- the ``trivial'' ones, of the form $(x^2+x)^{2^n-1}$, for some positive integer $n$,\\
- nine others which are the unique even and only divisible by Mersenne primes (\cite{Gall-Rahav-Mers}, Theorem 1.1),\\
- and two ones which are divisible by a non Mersenne prime (\cite{Gall-Rahav13}).

We are unable to find odd perfect polynomials. However, one gets the following results for such a polynomial $A$:\\
- $A$ is a square: $A = S^2$. Furthermore, we say that $A$ is \emph{special} if $S$ is square-free (\cite{Gall-Rahav4}),\\
- the number of irreducible factors of $A$, counted with multiplicity, is at least $12$ and $\deg(A) > 200$ (\cite{Pollack}),\\
- $\omega(A) \geq 5$ and if $A$ is special, then $\omega(A) \geq 10$ (\cite{Gall-Rahav4} and \cite{Gall-Rahav7}).

\section{Multiplicative function over $\F_2$}
\begin{definition} \label{defmultip}
\emph{Let $f: \F_2[x] \setminus {0} \to \F_2[x]$ be a map.
It is said to be \emph{multiplicative} (resp. \emph{totally
multiplicative}) if $f(AB)= f(A)f(B)$ whenever $\gcd(A,B) = 1$
(resp. for any $A, B \in \F_2[x] \setminus {0}$).}
\end{definition}
\begin{lemma} \label{propertmultip}
Let $f$ be a multiplicative function. Then, $f(1) = 1$ and $f$ is completely determined by the values of $f(P^r)$, for $P$ irreducible and $r \in \N^*$.
\end{lemma}
\begin{examples} \label{multipexamples}~\\
\emph{
$\bullet$ Multiplicative identity $\delta$:
$\delta(A) = 1$ if $A = 1$, $\delta(A) = 0$ otherwise.\\
$\bullet$ Constant function $z$: $z(A)=1$, for any
$A$ ($z$ is not the identity).\\
$\bullet$ Identity function $\ {\rm{Id}}$:  ${\rm{Id}}(A) = A$, for any $A$.\\
$\bullet$ The Euler function $\phi$: $\phi(P^r) = P^r + P^{r-1} \text{ if $P$ is irreducible and $r \geq 1$}$.\\
$\bullet$ The function $\sigma$: $\sigma(A)$ equals the sum of all divisors of $A$.\\
$\bullet$ The function $\sigma^*$: $\sigma^*(A)$ equals the sum of all unitary divisors of $A$.\\
$\bullet$ The ${\rm{M\ddot{o}bius}}$ function $\ \mu$:
$\mu(A) = \left\{\begin{array}{l} 1 \mbox{ if $A$ is square-free},\\
0 \mbox{ otherwise}.
\end{array}
\right.$\\
Note that $\delta, \ z, \ {\rm{Id}}, \ {\rm{Id}}_k$ are all totally
multiplicative whereas $\mu$, $\sigma$ and $\sigma^*$ are
only multiplicative.}
\end{examples}

\section{Dirichlet Convolution}
\begin{definition}
For two multiplicative functions $f, g$, we define the convolution
product as: $$(f * g) (A) = \displaystyle{\sum_{D
\mid A}} f(D) \ g(\frac{A}{D}).$$
\end{definition}
We get by direct computations the following lemmas and examples.
\begin{lemma} \label{convmultip}
The convolution $f * g$ is also multiplicative. Moreover,
$$\begin{array}{l}
f * g = g* f, \ f*(g*h)=(f*g)*h, \ f * \delta = f, \ f*(g+h)=f*g +
f*h,\\
\mbox{and $f(g*h) = fg*fh $ if $f$ is totally multiplicative}.
\end{array}$$
\end{lemma}
\begin{lemma} \label{convexamples} The following equalities hold:
$$z * \mu = \delta, \ \phi * z = {\rm{Id}}, \ \sigma=  {\rm{Id}} * z.$$
\end{lemma}
\begin{proof} Consider the value at $P^r$, for $P$ irreducible and $r \in \N^*$.
\end{proof}

\begin{lemma} [${\rm{M\ddot{o}bius}}$ {\bf{inversion formula}}] \label{Mobius}~\\
One has $g = f * z$ if and only if $f = g*\mu$.
\end{lemma}

\begin{remark}
\emph{If $f,g$ and $h$ are all multiplicative, with $f(S)=g(S)$, for some $S$, then in general, $(h *f)(S) \not= (h*g)(S)$.\\
For example, $S = x(x+1),\ f = {\rm{Id}}, \ h=g = \sigma$. One has: $\sigma(S) = S = {\rm{Id}}(S)$ but $(\sigma * \sigma)(S) = 0 \not= 1 = (\sigma * {\rm{Id}})(S)$ (see Lemmas \ref{squareconv} and \ref{convsigmaId})}.
\end{remark}
\subsection{The square convolution}
\begin{definition}
The {\it{square convolution}} of $f$, denoted by $f^{\rm{2conv}}$, is the convolution $f*f$.
\end{definition}
\begin{lemma} \label{squareconv} Let $P$ be irreducible. Then for any $r \geq 0$,
$f^{\rm{2conv}}(P^{2r})= (f(P^r))^2$ and $f^{\rm{2conv}}(P^{2r+1})= 0$.
\end{lemma}
\begin{proof}
First, $f^{\rm{2conv}}(P^{0})= 1$. For $r \geq 1$, \\
$\begin{array}{lll}
f^{\rm{2conv}}(P^{2r}) &=& \dis{\sum_{t=0}^{2r} f(P^{t}) \cdot f(P^{2r-t})}
=\dis{f(P^{2r}) + f(P^{2r})+\sum_{t=1}^{2r-1} f(P^{t}) \cdot f(P^{2r-t})}\\
&=&\dis{\sum_{t=1}^{r-1} f(P^{t}) f(P^{2r-t}) + (f(P^r))^2 + \sum_{t=r+1}^{2r-1} f(P^{t})  f(P^{2r-t})}
=(f(P^r))^2.
\end{array}$\\
\\
$\begin{array}{lll}
f^{\rm{2conv}}(P^{2r+1}) &=& \dis{\sum_{t=0}^{2r+1} f(P^{t}) \cdot f(P^{2r+1-t})}
=\dis{0+\sum_{t=1}^{2r} f(P^{t}) \cdot f(P^{2r+1-t})}\\
&=&\dis{\sum_{t=1}^{r} f(P^{t}) \cdot f(P^{2r+1-t}) + \sum_{t=r+1}^{2r} f(P^{t}) \cdot f(P^{2r+1-t})}
=0.
\end{array}$
\end{proof}
We immediately get
\begin{corollary} \label{corol-squareconv}
Let $A$ be a nonzero binary polynomial. Then
$f^{\rm{2conv}}(A)= A$ if and only if $A = S^2$ and $f(S) = S$.
\end{corollary}
\subsection{The convolution $\sigma * \mu$}
\begin{lemma} \label{convsigmamu} Let $P$ be irreducible and $m \in \N^*$. Then $$\text{$(\sigma * \mu)(P^{m})=P^m$ so that $\sigma * \mu = {\rm{Id}}$.}$$
\end{lemma}
\begin{proof} One has $(\sigma * \mu)(P)= \sigma(P) + \mu(P) = 1+P + 1 = P$ and
for $m \geq 2$, $(\sigma * \mu)(P^{m})= \sigma(P^m) + \mu(P^m) + \sigma(P^{m-1}) \mu(P) + 0 \cdots+0 =
\sigma(P^m) + \sigma(P^{m-1}) = P^m$.
\end{proof}
\begin{corollary} \label{corol-convsigmamu}
If $A$ is odd and perfect, then
$$\dis{\sum_{D \mid A,D\not=1,A, A/D \ square-free} \sigma(D) =0}.$$
\end{corollary}
\begin{proof}
One has: $\sigma(A) = A$ and $A = S^2$. Thus, \\
$\dis{A+0 + \sum_{D \mid A,D\not=1,A, A/D \ square-free} \sigma(D) = (\sigma * \mu)(A) = {\rm{Id}}(A) =A}$.
\end{proof}
\subsection{The convolution $\sigma * z$}
\begin{lemma} \label{convsigmaz} Let $P$ be irreducible and $r \in \N$. Then $$\text{$(\sigma * z)(P^{2r})=(\sigma(P^r))^2$ and $(\sigma * z)(P^{2r+1})
=  P \cdot (\sigma(P^r))^2$.}$$
\end{lemma}
\begin{proof} By induction on $r$. The case $r=0$ is trivial. \\
Suppose that $(\sigma * z)(P^{2r})=(\sigma(P^r))^2$ and $(\sigma * z)(P^{2r+1})
=  P \cdot (\sigma(P^r))^2$.\\
One has: $\displaystyle{(\sigma * z)(P^{2r+2})= \sigma(P^{2r+2}) + z(P^{2r+2}) + \sum_{k=1}^{2r+1} \sigma(P^{2r+2-k}) \cdot z(P^k)}$.\\
Thus, $\displaystyle{(\sigma * z)(P^{2r+2}) = \sigma(P^{2r+2}) + 1 + \sum_{k=1}^{2r+1} \sigma(P^{2r+2-k}) = (1+\cdots+P^{r+1})^2}$.\\
$\displaystyle{(\sigma * z)(P^{2r+3})= \sigma(P^{2r+3}) + z(P^{2r+3}) + \sum_{k=1}^{2r+2} \sigma(P^{2r+3-k}) \cdot z(P^k)}$.\\
So, $\displaystyle{(\sigma * z)(P^{2r+3}) = \sigma(P^{2r+3}) + 1 + \sum_{k=1}^{2r+2} \sigma(P^{2r+3-k}) = P \cdot (1+\cdots+P^{r+1})^2}$.
\end{proof}
\begin{corollary} \label{corol-convsigmaz}
If $A$ is special and perfect, then $\dis{\sum_{D \mid A,D\not=1,A} \sigma(D) = A+1+\sigma^*(A)}$.
\end{corollary}
\begin{proof}
One has: $\sigma(A) = A$ and $A = S^2$, with $S$ square-free. Thus, \\
$\dis{A+1 + \sum_{D \mid A,D\not=1,A} \sigma(D) = (\sigma * z)(A) = \sigma^*(A)}$.
\end{proof}
\subsection{The convolution $\sigma * {\rm{Id}}$}
\begin{lemma}  \label{convsigmaId} Let $P$ be irreducible. Then for any $r \geq 0$, $$(\sigma * {\rm{Id}})(P^{2r})=(\sigma * {\rm{Id}})(P^{2r+1})
=  (\sigma(P^r))^2.$$
\end{lemma}
\begin{proof} 
One has:\\
$\begin{array}{lcl}
(\sigma * {\rm{Id}})(P^m) &=& \dis{\sum_{\ell=0}^{m} \sigma(P^{\ell}) \cdot P^{m-\ell}}\\
&=&\dis{P^m+ \sigma(P^m)+ \sum_{\ell=1}^{m-1}  \sigma(P^{\ell}) \cdot P^{m-\ell}}\\
&=&\dis{\sigma(P^{m-1})+ \frac{1}{1+P} \cdot  \sum_{\ell=1}^{m-1}  (1+P^{\ell+1}) P^{m-\ell}}\\
&=&\dis{\sigma(P^{m-1})+ \frac{(m-1)P^{m+1}}{1+P} +  \frac{P}{1+P} \cdot  \frac{1+P^{m-1}}{1+P}}\\
&=&\dis{\frac{(1+P^2) \sigma(P^{m-1})+ (1+P) (m-1)P^{m+1}+P+P^m}{1+P^2}}\\
&=&\dis{\frac{1+mP^{m+1}+ (m+1)P^{m+2}}{1+P^2}}.
\end{array}$\\
\\
We get our results if we take $m=2r$ or $m=2r+1$.
\end{proof}
\begin{corollary} \label{corol-convsigmaId}
If $A=S^2$ is odd and perfect, then $\dis{\sum_{D \mid A,D\not=1,A} \frac{\sigma(D)}{D} = \frac{(\sigma(S))^2}{S^2}}$.
\end{corollary}
\begin{proof}
One has $A = S^2$, $0 = A + \sigma(A)$ and $(\sigma(S))^2 = (\sigma * {\rm{Id}})(S^2) = \sigma(A)+A+\dis{\sum_{D \mid A,D\not=1,A} \sigma(D) \cdot \frac{A}{D}}$.
\end{proof}
\subsection{The convolution $\sigma * \phi$}
\begin{lemma} \label{convsigmaphi} Let $P$ be irreducible. Then for any $r \geq 0$, $(\sigma * \phi)(P^{2r})= P^{2r}$ and $(\sigma * \phi)(P^{2r+1})= 0$.
\end{lemma}
\begin{proof}
By direct computations, as above, one has:
$$(\sigma * \phi)(P^m) = \dis{\sum_{\ell=0}^{m} \sigma(P^{\ell}) \cdot \phi(P^{m-\ell}) = (m-1)P^m}.$$
\end{proof}
\begin{corollary}\label{corol-convsigmaphi}
If $A$ is odd and perfect, then
$\phi(A)=\dis{\sum_{D\mid A, D \not=1,A} \sigma(D) \phi(\frac{A}{D})}$.
\end{corollary}
\begin{proof}
First, $A$ must be a square. So, we get $$\dis{A =(\sigma * \phi)(A)= \sigma(A) + \sum_{D\mid A, D \not=A} \sigma(D) \phi(A/D) =
A + \sum_{D\mid A, D \not=A} \sigma(D) \phi(A/D)}.$$
\end{proof}
\subsection{The convolution $\sigma^* * \mu$}
\begin{lemma} \label{convsigmastarmu} Let $P$ be irreducible and $m \in \N^*$. Then
$$\text{$(\sigma^* * \mu)(P)=P$ and $(\sigma^* * \mu)(P^{m})=P^{m}+P^{m-1} = \phi(P^m)$ if $m \geq 2$.}$$
\end{lemma}
\begin{proof} 
The case $m=1$ is trivial. For $m\geq 2$, one has
$(\sigma^* * \mu)(P^{m})=\sigma^*(P^{m}) + \mu(P^m) + \sigma^*(P^{m-1}) \mu(P) =
\sigma^*(P^{m}) + \sigma^*(P^{m-1}) = P^m+P^{m-1}$.
\end{proof}
\begin{corollary} \label{corol-convsigmastarmu}
If $A$ is a square, then $\dis{\phi(A)= \sum_{D \mid A, A/D \ square-free} \sigma^*(D)}$.
\end{corollary}
\begin{proof}
One has  $A = S^2$. Thus,
$\dis{\sum_{D \mid A, \ A/D \ square-free} \sigma^*(D) = (\sigma^* * \mu)(A) = \phi(A)}$.
\end{proof}
\subsection{The convolution $\sigma^* * z$}
\begin{lemma} \label{convsigmastarz} Let $P$ be irreducible and $r \in \N$. Then
$$\text{$(\sigma^* * z)(P^{2r})=\sigma(P^{2r})$ and $(\sigma^* * z)(P^{2r+1})
=  P \cdot \sigma(P^{2r})$.}$$
\end{lemma}
\begin{proof} 
By induction on $r$. The case $r=0$ is trivial. \\
Suppose that $(\sigma^* * z)(P^{2r})=\sigma(P^{2r})$ and $(\sigma^* * z)(P^{2r+1})
=  P \cdot \sigma(P^{2r})$.\\
One has: $\displaystyle{(\sigma^* * z)(P^{2r+2})= \sigma^*(P^{2r+2}) + z(P^{2r+2}) + \sum_{k=1}^{2r+1} \sigma^*(P^{2r+2-k}) \cdot z(P^k)}$.\\
Thus, $\displaystyle{(\sigma^* * z)(P^{2r+2}) = \sigma^*(P^{2r+2}) + 1 + \sum_{k=1}^{2r+1} \sigma^*(P^{2r+2-k}) = \sigma(P^{2r+2})}$,\\
$\displaystyle{(\sigma^* * z)(P^{2r+3})= \sigma^*(P^{2r+3}) + z(P^{2r+3}) + \sum_{k=1}^{2r+2} \sigma^*(P^{2r+3-k}) \cdot z(P^k)}$.\\
So, $\displaystyle{(\sigma^* * z)(P^{2r+3}) = \sigma^*(P^{2r+3}) + 1 + \sum_{k=1}^{2r+2} \sigma^*(P^{2r+3-k}) = P \cdot \sigma(P^{2r+2})}$.
\end{proof}
\begin{corollary} \label{corol-convsigmastarz}
If $A$ is special and perfect, then $\dis{A= \sum_{D \mid A} \sigma^*(D)}$.
\end{corollary}
\begin{proof}
One has: $\sigma(A) = A$ and $A = S^2$, with $S$ square-free. Thus, \\
$\dis{\sum_{D \mid A} \sigma^*(D) = (\sigma^* * z)(A) = \sigma(A) = A}$.
\end{proof}
\subsection{The convolution $\sigma^* * {\rm{Id}}$}
\begin{lemma} \label{convsigmastarId} Let $P$ be irreducible and $r \in \N$. Then
$$\text{$(\sigma^* * {\rm{Id}})(P^{2r})=\sigma(P^{2r}) = (\sigma^* * {\rm{Id}})(P^{2r+1})$.}$$
\end{lemma}
\begin{proof} 
For $m \geq 1$, one has after computations:
$$(\sigma^* * {\rm{Id}})(P^m) = \dis{\sum_{\ell=0}^{m} \sigma^*(P^{\ell}) \cdot P^{m-\ell} = 1+ (m-1)P^m+(P+\cdots+P^{m-1})}.$$
We get our result if $m$ is even (resp. if $m$ is odd).
\end{proof}
\begin{corollary} \label{corol-convsigmastarId}
If $A$ is odd and perfect, then $\dis{\sum_{D\mid A, D \not=1,A} \frac{\sigma^*(D)}{D} = \frac{\sigma^*(A)}{A}}$.
\end{corollary}
\begin{proof}
First, $A$ must be a square. So, we get $$A=\dis{\sigma(A) =(\sigma^* * {\rm{Id}})(A)= \sigma^*(A) + A+ \sum_{D\mid A, D \not=1,A} \sigma^*(D) \cdot \frac{A}{D}}.$$
\end{proof}
\subsection{The convolution $\sigma^* * \phi$}
\begin{lemma} \label{convsigmastarphi} Let $P$ be irreducible and $r \in \N$. Then
$$\text{$(\sigma^* * \phi)(P^{2r})= \phi(P^{2r})$ and $(\sigma^* * \phi)(P^{2r+1})= 0$.}$$
\end{lemma}
\begin{proof} 
One has:
$(\sigma^* * \phi)(P^m) = \dis{\sum_{\ell=0}^{m} \sigma^*(P^{\ell}) \cdot \phi(P^{m-\ell}) = (m-1)(P^m + P^{m-1}).}$
\end{proof}
\begin{corollary} \label{corol-convsigmastarphi}
If $A$ is odd and perfect, then $\dis{\sigma^*(A)=\sum_{D\mid A, D \not=1,A} \sigma^*(D) \cdot \phi(\frac{A}{D})}$.
\end{corollary}
\subsection{The convolution $\sigma^* * \sigma$}
\begin{lemma} \label{convsigmastarsigma}Let $P$ be irreducible and $r \in \N$. Then
$$\text{$(\sigma^* * \sigma)(P^{2r})=
\sigma(P^{2r})$ and $(\sigma^* * \sigma)(P^{2r+1})=0$}.$$
\end{lemma}
\begin{proof}
One has:
$(\sigma^* * \sigma)(P^m) = \dis{\sum_{\ell=0}^{m} \sigma^*(P^{\ell}) \cdot \sigma(P^{m-\ell}) = (m-1) \sigma(P^m).}$
\end{proof}
\begin{corollary} \label{corol-convsigmastarsigma}
If $A$ is odd and perfect, then $\dis{\sigma^*(A)=\sum_{D\mid A, D \not=1,A} \sigma^*(D) \cdot \sigma(\frac{A}{D})}$.
\end{corollary}
\begin{proof}
The polynomial $A$ is a square. Thus, $A = \sigma(A) = (\sigma^* * \sigma)(A)$.
\end{proof}
\section{The Dirichlet Inverse}
\begin{lemma} \label{definvconv}
Let $f$ be a multiplicative function. Then,
there exists a unique multiplicative function $f^{\rm{inv}}$
$($called the Dirichlet inverse of $f)$ such that $f*f^{\rm{inv}} = \delta$.
\end{lemma}
\begin{proof}
Set $f^{\rm{inv}}(1) = f(1)=1$. If $A=P^r$, with $P$ irreducible, we recursively define $f^{\rm{inv}}(P^r)$ by putting:
$f^{\rm{inv}}(P)= f(P), \
f^{\rm{inv}}(P^r)=\dis{\sum_{\ell = 0}^{r-1} f^{\rm{inv}}(P^{\ell})f(P^{r-\ell})}.
$\\
If $A = P_1^{r_1} \cdots
P_k^{r_k}$, with each $P_i$ irreducible and $P_i \not= P_j$, if $i \not= j$, then we define $f^{\rm{inv}}(A)$ as  $f^{\rm{inv}}(P_1^{r_1}) \cdots f^{\rm{inv}}(P_k^{r_k})$.
\end{proof}
\begin{lemma} \label{propert-invconv}
Let $f$ and $g$ be multiplicative. Then
$$\text{$(f^{\rm{inv}})^{{\rm{inv}}} = f$ and  $(f*g)^{\rm{inv}} = f^{\rm{inv}} * g^{\rm{inv}}$.}$$
\end{lemma}
\begin{examples} \label{convexamples2} {\emph{From Lemma \ref{convexamples}, we immediately get}}\\
$$\text{$\mu^{\rm{inv}}=z, \ z^{\rm{inv}}= \mu$, $\phi^{\rm{inv}} * {\rm{Id}}=z$, ${\rm{Id}}^{\rm{inv}} = \phi^{\rm{inv}} * \mu$, ${\rm{Id}}* \mu = \phi$.}$$
\end{examples}
\subsection{The Dirichlet inverse of ${\rm{Id}}$} \label{invId}
We denote by ${\rm{Id}}^{\rm{inv}}$ the Dirichlet inverse of ${\rm{Id}}$.
\begin{lemma} \label{express-invsigma}
If $P$ is irreducible, then ${\rm{Id}}^{\rm{inv}}(P)=P$ and ${\rm{Id}}^{\rm{inv}}(P^m)=0$, for any $m \geq 2$.
\end{lemma}
\begin{proof}
By induction on $m$. First, ${\rm{Id}}^{\rm{inv}}(P)= {\rm{Id}}(P)=P$ and ${\rm{Id}}^{\rm{inv}}(P^{2}) = {\rm{Id}}(P^2) + {\rm{Id}}^{\rm{inv}}(P) \cdot {\rm{Id}}(P) = P^2 + P \cdot P = 0$. Suppose that ${\rm{Id}}^{\rm{inv}}(P^{\ell})=0$, for $2 \leq \ell \leq m-1$.  We get:
$$\begin{array}{lcl}
{\rm{Id}}^{\rm{inv}}(P^{m})&=&\dis{{\rm{Id}}(P^{m}) + \sum_{\ell =1}^{m-1}
{\rm{Id}}^{\rm{inv}}(P^{\ell}) \cdot {\rm{Id}}(P^{m-\ell})}\\
&=& \dis{P^{m}+ P \cdot P^{m-1} + \sum_{\ell =2}^{m-1} {\rm{Id}}^{\rm{inv}}(P^{\ell}) \cdot {\rm{Id}}(P^{m-\ell})}\\
&=& P^m+P^{m}+0.
\end{array}$$
\end{proof}
\subsubsection{The convolution ${\rm{Id}}^{\rm{inv}} * z$}
\begin{lemma} Let $P$ be irreducible and $m \geq 1$. Then,
$({\rm{Id}}^{\rm{inv}} * z)(P^{m})
=  1+P$.
\end{lemma}
\begin{proof} 
By induction on $m$: $({\rm{Id}}^{\rm{inv}} * z)(P)= {\rm{Id}}^{\rm{inv}}(P) + z(P) = P+1$.\\
For $m\geq 2$:\\
$\begin{array}{lcl}
({\rm{Id}}^{\rm{inv}} * z)(P^m)&=&\dis{{\rm{Id}}^{\rm{inv}}(P^m) + z(P^m) + \sum_{k=1}^{m-1} {\rm{Id}}^{\rm{inv}}(P^{m-k}) \cdot z(P^k)}\\
&=& 0+1+0+ {\rm{Id}}^{\rm{inv}}(P) \cdot z(P^{m-1}).
\end{array}$
\end{proof}
\subsubsection{The convolution $\sigma * {\rm{Id}}^{\rm{inv}}$}
\begin{lemma} One has $\sigma * {\rm{Id}}^{\rm{inv}} = z$ so that
$(\sigma * {\rm{Id}}^{\rm{inv}})(P^{m})= 1$ for any irreducible $P$ and $m \geq 0$.
\end{lemma}
\begin{proof}
It follows from the facts: $\sigma * \mu = {\rm{Id}}$ and $\mu^{\rm{inv}} = z$.
\end{proof}
\begin{corollary}
If $A$ is odd and perfect, then $$\text{$\dis{\frac{A+1}{A} \ = \sum_{D\mid A, D \not=A, \frac{A}{D} \text{square-free}} \frac{\sigma(D)}{D}}$.}$$
\end{corollary}
\subsection{The Dirichlet inverse of $\phi$} \label{invphi}
We denote by $\phi^{\rm{inv}}$ the Dirichlet inverse of $\phi$.
\begin{lemma} \label{express-invsigma}
One has $\phi^{\rm{inv}} = {\rm{Id}}^{\rm{inv}} *z$ so that $\phi^{\rm{inv}}(P^m)=1+P,$ for any irreducible $P$ and $m \in \N^*$.
\end{lemma}
\begin{proof}
It follows from the facts: $\phi = {\rm{Id}} * \mu$ and $\mu^{\rm{inv}} = z$.
\end{proof}
\subsubsection{The convolution $\sigma * \phi^{\rm{inv}}$}
\begin{lemma} Let $P$ be irreducible and $r \in \N$. Then
$$\text{$(\sigma * \phi^{\rm{inv}})(P^{2r})= 1$ and $(\sigma * \phi^{\rm{inv}})(P^{2r+1})= 0$.}$$
\end{lemma}
\begin{proof}
One has:\\
$\begin{array}{lcl}
(\sigma * \phi^{\rm{inv}})(P^m) &=& \dis{\sum_{\ell=0}^{m} \sigma(P^{\ell}) \cdot \sigma^{\rm{inv}}(P^{m-\ell})}\\
&=&\dis{\phi^{\rm{inv}}(P^m)+ \sigma(P^m)+ \sum_{\ell=1}^{m-1} \sigma(P^{\ell}) \cdot \phi^{\rm{inv}}(P^{m-\ell})}\\
&=&\dis{1+P+\sigma(P^m)+ \sum_{\ell=1}^{m-1} (1+P) \cdot
\sigma(P^{\ell})}\\
&=&\dis{1+P+\sigma(P^m)+ \sum_{\ell=1}^{m-1} (1+P^{\ell+1})}\\
&=& m-1.
\end{array}$\\
\end{proof}
\begin{corollary}
If $A$ is odd and perfect, then $$\dis{A= 1+\sigma(rad(A)) + \sum_{D\mid A, D \not=1,A} \sigma(D) \cdot \phi^{\rm{inv}}(\frac{A}{D})}.$$
\end{corollary}
\subsection{The Dirichlet inverse of $\sigma$} \label{invsigma}
We denote by $\sigma^{\rm{inv}}$ the Dirichlet inverse of $\sigma$.
\begin{lemma} \label{express-invsigma}
If $P$ is irreducible, then
$$\text{$\sigma^{\rm{inv}}(P)=1+P$,  $\sigma^{\rm{inv}}(P^2)=P$ and  $\sigma^{\rm{inv}}(P^m)=0$ for $m \geq 3$.}$$
\end{lemma}
\begin{proof}
$\sigma^{\rm{inv}}(P)= \sigma(P)=1+P$, $\sigma^{\rm{inv}}(P^2)= \sigma(P^2) + \sigma^{\rm{inv}}(P) \sigma(P)= P.$\\
$\begin{array}{lcl}
\sigma^{\rm{inv}}(P^3)&= &\sigma(P^3) + \sigma^{\rm{inv}}(P) \sigma(P^2) + \sigma^{\rm{inv}}(P^2) \sigma(P)\\
  &=& 1+P+P^2+P^3 + (1+P)(1+P+P^2)+P(1+P)\\
  &=&0.
  \end{array}$.\\
For $m \geq 4$, we proceed by induction on $m$. Suppose that $\sigma^{\rm{inv}}(P^{\ell}) = 0$, for $3 \leq \ell \leq m-1$. We get:\\
$\begin{array}{lcl}
\sigma^{\rm{inv}}(P^{m})&=&\dis{\sigma(P^{m}) + \sum_{\ell =1}^{m-1}
 \sigma^{\rm{inv}}(P^{\ell})  \sigma(P^{m-\ell})}\\
&=& \dis{\sigma(P^{m})+ (1+P) \sigma(P^{m-1})+P \sigma(P^{m-2})+0}\\
&=&  \sigma(P^{m}) +  \sigma(P^{m-1}) + P(\sigma(P^{m-1}) + \sigma(P^{m-2})) \\
&=& P^m + P \cdot P^{m-1}\\
&=& 0.
\end{array}$
\end{proof}
\begin{corollary} \label{perfectcondition}
If $A$ is odd and perfect, then $$\dis{A = \sigma^{\rm{inv}}(A) + \sum_{D \mid A, D \not=A,1} \sigma^{\rm{inv}}(D) \cdot \sigma(\frac{A}{D})}.$$
\end{corollary}
\begin{proof}
It follows from the facts: $(\sigma^{\rm{inv}} * \sigma)(A) = \delta(A) = 0$ and $\sigma(A)=A$.
\end{proof}
\subsubsection{The convolution $\sigma^{\rm{inv}} * z$}
\begin{lemma} One has $\sigma^{\rm{inv}} * z = {\rm{Id}}^{\rm{inv}}$ so that for any irreducible $P$,
$$\text{$(\sigma^{\rm{inv}} * z)(P)=P$ and $(\sigma^{\rm{inv}} * z)(P^m)=0$ if $m \geq 2$.}$$
\end{lemma}
\begin{proof}
We get $\sigma^{\rm{inv}} * z = (\sigma * \mu)^{\rm{inv}} = {\rm{Id}}^{\rm{inv}}$.
\end{proof}
\subsubsection{The convolution $\sigma^{\rm{inv}} * {\rm{Id}}$}
\begin{lemma} One has $\sigma^{\rm{inv}} * {\rm{Id}}=\mu$ so that for any irreducible  $P$,
$$\text{$(\sigma^{\rm{inv}} * {\rm{Id}})(P)=1$ and $(\sigma^{\rm{inv}} * {\rm{Id}})(P^{m})=0$, for any $m \geq 2$.}$$
\end{lemma}
\begin{proof} It follows from the fact: $\sigma * \mu = {\rm{Id}}$.
\end{proof}
\begin{corollary}
If $A$ is special and perfect, then $$\dis{A = rad(A)+ \sum_{D \mid A, A \not= A,1} \sigma^{\rm{inv}}(D) \cdot \frac{A}{D}}.$$
\end{corollary}
\begin{proof}
It follows from the fact: $(\sigma^{\rm{inv}} * {\rm{Id}})(A)= 0$.
\end{proof}
\subsubsection{The convolution $\sigma^{\rm{inv}} * \mu = (\sigma * z)^{\rm{inv}}$}
\begin{lemma} Let $P$ be irreducible. Then $(\sigma^{\rm{inv}} * \mu)(P^m)=P$ if $m \in \{1,3\}$,  $(\sigma^{\rm{inv}} * \mu)(P^2)=1$ and $(\sigma^{\rm{inv}} * \mu)(P^m)=0$ if $m \geq 4$.
\end{lemma}
\begin{proof} 
$(\sigma^{\rm{inv}} * \mu)(P) =\sigma^{\rm{inv}}(P) + \mu(P) =1+P+1 = P$.\\
$(\sigma^{\rm{inv}} * \mu)(P^2) =\sigma^{\rm{inv}}(P^2) + \mu(P^2) +\sigma^{\rm{inv}}(P) \cdot \mu(P) =P+0+(1+P) \cdot 1 = 1$.\\
$(\sigma^{\rm{inv}} * \mu)(P^3) = 0 + 0 + P \cdot 1 + 0 = P$.\\
For $m\geq 4$,
$(\sigma^{\rm{inv}} * \mu)(P^m) = \sigma^{\rm{inv}}(P^m)+ \mu(P^m)+ \sigma^{\rm{inv}}(P^{m-1}) \cdot \mu(P) + 0 = 0 +0$.\\
\end{proof}
\begin{corollary}
If $A$ is special and perfect, then $$\dis{\sum_{D \mid A, A \not= A,1, \text{$\frac{A}{D}$ square-free}} \sigma^{\rm{inv}}(D) = 1 + rad(A)}.$$
\end{corollary}
\begin{proof}
It follows from the facts: $(\sigma^{\rm{inv}} * \mu)(A)= 1$ and $\sigma^{\rm{inv}}(A) = rad(A)$.
\end{proof}
\subsubsection{The convolution $\sigma^* * \sigma^{\rm{inv}}$}
\begin{lemma} Let $P$ be irreducible and $m \geq 1$. Then $$\text{$(\sigma^* * \sigma^{\rm{inv}})(P^2)= P$ and
$(\sigma^* * \sigma^{\rm{inv}})(P^m)= 0$ if $m \not= 2$.}$$
\end{lemma}
\begin{proof}
One has:\\
$\begin{array}{lcl}
(\sigma^* * \sigma^{\rm{inv}})(P^m) &=& \dis{\sum_{\ell=0}^{m} \sigma^*(P^{\ell}) \cdot \sigma^{\rm{inv}}(P^{m-\ell})}\\
&=&\dis{\sigma^{\rm{inv}}(P^m)+ \sigma^*(P^m)+ \sum_{\ell=1}^{m-1} \sigma^*(P^{\ell}) \cdot \sigma^{\rm{inv}}(P^{m-\ell})}.
\end{array}$\\
Recall that  $\sigma^{\rm{inv}}(P) =1+P$,  $\sigma^{\rm{inv}}(P^2) =P$  and $\sigma^{\rm{inv}}(P^m) =0$ if $m \geq 3$.\\
If $m \in \{1,2\}$, then we get our results by direct computations. \\
For $m\geq 3$,
$\sigma^{\rm{inv}}(P^m) = 0$ and $\sigma^{\rm{inv}}(P^{m-\ell}) =0$, if $m-\ell \geq 3$.\\
Therefore,\\
$\begin{array}{lcl}
(\sigma^* * \sigma^{\rm{inv}})(P^m) &=&\dis{1+P^m + \sum_{\ell=m-2}^{m-1} (1+P^{\ell}) \cdot
\sigma^{\rm{inv}}(P^{m-\ell})}\\
&=&\dis{1+P^m+ (1+P^{m-2}) \cdot P + (1+P^{m-1}) \cdot (1+P)}\\
&=&0.
\end{array}$\\
\end{proof}
\subsection{The Dirichlet inverse of $\sigma^*$} \label{invsigmastar}
We denote by $\sigma^{*\rm{inv}}$ the Dirichlet inverse of $\sigma^*$.
\begin{lemma} \label{express-invsigma}
Let $P$ be irreducible and $r \in \N$. Then
$$\text{$\sigma^{*\rm{inv}}(P^{2r})=0$ and $\sigma^{*\rm{inv}}(P^{2r+1})=P^{r}(1+P)$.}$$
\end{lemma}
\begin{proof}
We prove the statement by induction on $r$. The case $r=0$ is trivial. For $0 \leq t \leq r-1$, suppose that
$\text{$\sigma^{*\rm{inv}}(P^{2t})=0$ and $\sigma^{*\rm{inv}}(P^{2t+1})=P^{t}(1+P)$.}$ \\
One has:
$$\begin{array}{lcl}
\sigma^{*\rm{inv}}(P^{2r})&=&\dis{\sigma^{*}(P^{2r}) + \sigma^{*\rm{inv}}(P)  \sigma^{*}(P^{2r-1}) + \sum_{t =1}^{r-1}
 \sigma^{*\rm{inv}}(P^{2t +1})  \sigma^{*}(P^{2r-2t-1})}\\
&=& \dis{1+P^{2r}+ (1+P)(1+P^{2r-1})+\sum_{t =1}^{r-1} P^{t}(1+P) (1+P^{2r-2t-1})}\\
&=& P+P^{2r-1}+(1+P)(P+P^2+\cdots+P^r+P^{r+1}+\cdots+P^{2r-2})\\
&=& P+P^{2r-1}+ P(1+P^{2r-2}) = 0.
\end{array}$$
Now, \\
$\begin{array}{lcl}
\sigma^{*\rm{inv}}(P^{2r+1})&=&\dis{\sigma^{*}(P^{2r+1}) + \sigma^{*\rm{inv}}(P)  \sigma^{*}(P^{2r}) + \sum_{t =1}^{r-1}
 \sigma^{*\rm{inv}}(P^{2t +1})  \sigma^{*}(P^{2r-2t})}\\
&=&\dis{1+P^{2r+1}+ (1+P)(1+P^{2r})+\sum_{t =1}^{r-1} P^{t}(1+P) (1+P^{2r-2t})}\\
&=&\dis{P+P^{2r}+(1+P)\sum_{t =1}^{r-1} (P^{t}+P^{2r-t})}\\
&=& P+P^{2r}+(1+P)(P+\cdots+P^{r-1}+P^{r+1}+\cdots+P^{2r-1})\\
&=& P+P^{2r}+ P(1+P^{2r-1}) +(1+P)P^r = 0 + (1+P)P^r.
\end{array}$\\
\end{proof}
\subsubsection{The convolution $\sigma^{*\rm{inv}} * z = (\sigma^* * \mu)^{\rm{inv}}$}
\begin{lemma} Let $P$ be irreducible and $r \in \N$. Then
$$\text{$(\sigma^{*\rm{inv}} * z)(P^{2r})=P^r$ and $(\sigma^{*\rm{inv}} * z)(P^{2r+1}) =P^{r+1}$.}$$
\end{lemma}
\begin{proof} Recall that  $\sigma^{*\rm{inv}}(P^{2t}) =0$ and $\sigma^{*\rm{inv}}(P^{2t+1}) = P^{t}(1+P)$. One has $(\sigma^{*\rm{inv}} * z)(P) = \sigma^{*\rm{inv}}(P) + z(P) = 1+P + 1 = P$,
$(\sigma^{*\rm{inv}} * z)(P^2) = \sigma^{*\rm{inv}}(P^2) + z(P^2) + \sigma^{*\rm{inv}}(P) \cdot z(P)=0+1+(1+P) \cdot 1= P$ and $(\sigma^{*\rm{inv}} * z)(P^3) = \sigma^{*\rm{inv}}(P^3) + z(P^3) + 0+ \sigma^{*\rm{inv}}(P) \cdot z(P^2)=P(1+P)+1+(1+P) \cdot 1= P^2$.\\
For $r \geq 2$,\\
$\begin{array}{lcl}
(\sigma^{*\rm{inv}} * z)(P^{2r})& =& \sigma^{*\rm{inv}}(P^{2r}) + z(P^{2r}) + \dis{\sum_{t=0}^{r-1} \sigma^{*\rm{inv}}(P^{2t+1}) \cdot 1}\\
&=&0+1+ \dis{\sum_{t=0}^{r-1} P^t(1+P) = 1+1+P^r = P^r.}
\end{array}$
\\
$(\sigma^{*\rm{inv}} * z)(P^{2r+1}) = P^r(1+P)+1+\dis{\sum_{t=0}^{r-1} P^t(1+P) \cdot 1 = P^{r+1}}$.
\end{proof}
\subsubsection{The convolution $\sigma^{*\rm{inv}} * {\rm{Id}}$}
\begin{lemma} Let $P$ be irreducible and $r \in \N$. Then
$$\text{$(\sigma^{*\rm{inv}} * {\rm{Id}})(P^{2r})=P^r$ and $(\sigma^{*\rm{inv}} * {\rm{Id}})(P^{2r+1})=P^{r}$.}$$
\end{lemma}
\begin{proof} 
One has:\\
$(\sigma^{*\rm{inv}} * {\rm{Id}})(P^{2r}) = 0 +P^{2r}+ \dis{\sum_{t=0}^{r-1} \sigma^{*\rm{inv}}(P^{2t+1}) P^{2r-2t-1} = P^{2r} + (1+P) \sum_{t=0}^{r-1}  P^{2r-t-1} = P^r}.$
\\
$(\sigma^{*\rm{inv}} * {\rm{Id}})(P^{2r+1}) = \dis{P^r(1+P) + P^{2r+1} + (1+P) \sum_{t=0}^{r-1}  P^{2r-t} = P^r}$.
\end{proof}
\begin{corollary}
If $A$ is special and perfect, then
$$\text{$\dis{A = rad(A)+ \sum_{D \mid A, D \not= A,1, D \text{square-free}} \sigma^{*\rm{inv}}(D) \cdot \frac{A}{D}}$.}$$
\end{corollary}
\begin{proof}
It follows from the fact: $(\sigma^{*\rm{inv}} * {\rm{Id}})(A)= rad(A)$.
\end{proof}
\subsubsection{The convolution $\sigma^{*\rm{inv}} * \mu = (\sigma^* * z)^{\rm{inv}}$}
\begin{lemma} Let $P$ be irreducible and $r \geq 1$. Then
$$\text{$(\sigma^{*\rm{inv}} * \mu)(P^{2r}) = P^{r-1}(1+P)$ and $(\sigma^{*\rm{inv}} * \mu)(P^{2r+1})=P^{r} \cdot (1+P)$.}$$
\end{lemma}
\begin{proof} 
$(\sigma^{*\rm{inv}} * \mu)(P) =\sigma^{*\rm{inv}}(P) + \mu(P) =1+P+1 = P$.\\
$(\sigma^{*\rm{inv}} * \mu)(P^2) =\sigma^{*\rm{inv}}(P^2) + \mu(P^2) +\sigma^{*\rm{inv}}(P) \cdot \mu(P) =0+0+(1+P) \cdot 1 = 1+P$.\\
For $r \geq 2$, $(\sigma^{*\rm{inv}} * \mu)(P^{2r}) =\sigma^{*\rm{inv}}(P^{2r}) + \mu(P^{2r}) +\sigma^{*\rm{inv}}(P^{2r-1}) \cdot \mu(P) + 0 =0+0+P^{r-1} (1+P) \cdot 1$.\\
$(\sigma^{*\rm{inv}} * \mu)(P^{2r+1}) =\sigma^{*\rm{inv}}(P^{2r+1}) + \mu(P^{2r+1}) + 0 =P^{r} \cdot (1+P) + 0$.
\end{proof}
\begin{corollary}
If $A$ is special and perfect, then $$\dis{\sum_{D \mid A, A \not= A,1, \text{$\frac{A}{D}$ square-free}} \sigma^{*\rm{inv}}(D) = \sigma(rad(A))}.$$
\end{corollary}
\begin{proof}
It follows from the facts: $(\sigma^{*\rm{inv}} * \mu)(A)= \sigma(rad(A))$ and $\sigma^{*\rm{inv}}(A) =0 = \mu(A)$.
\end{proof}
\subsubsection{The convolution $\sigma * \sigma^{\rm{*inv}} = (\sigma^* * \sigma^{\rm{inv}})^{\rm{inv}}$}
\begin{lemma} Let $P$ be irreducible and $r \geq 0$. Then\\
$$\text{$(\sigma * \sigma^{\rm{*inv}})(P^{2r})= P^r$ and $(\sigma * \sigma^{\rm{*inv}})(P^{2r+1})= 0$.}$$
\end{lemma}
\begin{proof}
One has:\\
$\begin{array}{lcl}
(\sigma * \sigma^{\rm{*inv}})(P^m) &=& \dis{\sum_{\ell=0}^{m} \sigma(P^{\ell}) \cdot \sigma^{\rm{*inv}}(P^{m-\ell})}\\
&=&\dis{\sigma^{\rm{*inv}}(P^m)+ \sigma(P^m)+ \sum_{\ell=1}^{m-1} \sigma(P^{\ell}) \cdot \sigma^{\rm{*inv}}(P^{m-\ell})}.
\end{array}$\\
$\bullet$ Case where $m=2r$ is even\\
$\sigma^{\rm{*inv}}(P^m) = 0$ and $\sigma^{\rm{*inv}}(P^{m-\ell}) =0$, if $m-\ell$ is even.\\
Put $m-\ell =2r-\ell= 2s+1$.\\
$\begin{array}{lcl}
(\sigma * \sigma^{\rm{*inv}})(P^m) &=&\dis{\sigma(P^m) + \sum_{s=0}^{r-1} \sigma(P^{2r-2s-1}) \cdot
\sigma^{\rm{*inv}}(P^{2s+1})}\\
&=&\dis{\sigma(P^m) + \sum_{s=0}^{r-1} \sigma(P^{2r-2s-1}) \cdot
P^{s} (1+P)}\\
&=&\dis{\sigma(P^m) + \sum_{s=0}^{r-1} (1+P^{2r-2s}) \cdot
P^{s}}\\
&=&\dis{\sigma(P^{2r}) + \sum_{s=0}^{r-1} (P^s+P^{2r-s})}\\
&=&P^r.
\end{array}$\\
$\bullet$ Case where $m=2r+1$ is odd\\
$\sigma^{\rm{*inv}}(P^m) = P^r(1+P)$ and $\sigma^{\rm{*inv}}(P^{m-\ell}) =0$, if $m-\ell$ is even.\\
Put $m-\ell =2r+1-\ell= 2s+1$.\\
$\begin{array}{lcl}
(\sigma * \sigma^{\rm{*inv}})(P^m) &=&\dis{P^r(1+P) + \sigma(P^{2r+1}) + \sum_{s=0}^{r-1} \sigma(P^{2r-2s}) \cdot
\sigma^{\rm{*inv}}(P^{2s+1})}\\
&=&\dis{P^r + P^{r+1} + \sigma(P^{2r+1}) + \sum_{s=0}^{r-1} \sigma(P^{2r-2s}) \cdot
P^{s} (1+P)}\\
&=&\dis{P^r + P^{r+1} + \sigma(P^{2r+1}) + \sum_{s=0}^{r-1} (1+P^{2r-2s+1}) \cdot
P^{s}}\\
&=&\dis{P^r + P^{r+1} + \sigma(P^{2r+1}) + \sum_{s=0}^{r-1} (P^s+P^{2r+1-s})}\\
&=&0.
\end{array}$
\end{proof}
\begin{corollary}
If $A$ is special and perfect, then
$$\text{$\dis{A = rad(A)+ \sum_{D \mid A, D \not= A,1, D \text{square-free}} \sigma^{*\rm{inv}}(D) \cdot \sigma(\frac{A}{D})}$.}$$
\end{corollary}
\begin{proof}
It follows from the facts: $\sigma^{*\rm{inv}}(A) = 0$ and $(\sigma^{*\rm{inv}} *  \sigma)(A)= rad(A)$.
\end{proof}
\section{More results: $\sigma^{\rm{inv}} * \phi$, $\sigma^{*\rm{inv}} * \phi$ and $\phi * {\rm{Id}}$}
By similar computations, we get
\begin{lemma} Let $P$ be irreducible and $m\geq 1$. Then\\
i) $(\sigma^{\rm{inv}} * \phi)(P^2) = 1$ and $(\sigma^{\rm{inv}} * \phi)(P^m) = 0$ if $m \not= 2$.\\
ii) $(\sigma^{*\rm{inv}} * \phi)(P^{2m}) = \phi(P^m)$ and $(\sigma^{\rm{inv}} * \phi)(P^{2m-1}) = 0$.\\
iii) $(\phi * {\rm{Id}})(P)=1$ and $(\phi * {\rm{Id}})(P^{2m}) = P^{2m} = (\phi * {\rm{Id}})(P^{2m+1})$.
\end{lemma}

\def\biblio{\def\titrebibliographie{References}\thebibliography}
\let\endbiblio=\endthebibliography




\newbox\auteurbox
\newbox\titrebox
\newbox\titrelbox
\newbox\editeurbox
\newbox\anneebox
\newbox\anneelbox
\newbox\journalbox
\newbox\volumebox
\newbox\pagesbox
\newbox\diversbox
\newbox\collectionbox
\def\fabriquebox#1#2{\par\egroup
\setbox#1=\vbox\bgroup \leftskip=0pt \hsize=\maxdimen \noindent#2}
\def\bibref#1{\bibitem{#1}


\setbox0=\vbox\bgroup}
\def\auteur{\fabriquebox\auteurbox\styleauteur}
\def\titre{\fabriquebox\titrebox\styletitre}
\def\titrelivre{\fabriquebox\titrelbox\styletitrelivre}
\def\editeur{\fabriquebox\editeurbox\styleediteur}

\def\journal{\fabriquebox\journalbox\stylejournal}

\def\volume{\fabriquebox\volumebox\stylevolume}
\def\collection{\fabriquebox\collectionbox\stylecollection}
{\catcode`\- =\active\gdef\annee{\fabriquebox\anneebox\catcode`\-
=\active\def -{\hbox{\rm
\string-\string-}}\styleannee\ignorespaces}}
{\catcode`\-
=\active\gdef\anneelivre{\fabriquebox\anneelbox\catcode`\-=
\active\def-{\hbox{\rm \string-\string-}}\styleanneelivre}}
{\catcode`\-=\active\gdef\pages{\fabriquebox\pagesbox\catcode`\-
=\active\def -{\hbox{\rm\string-\string-}}\stylepages}}
{\catcode`\-
=\active\gdef\divers{\fabriquebox\diversbox\catcode`\-=\active
\def-{\hbox{\rm\string-\string-}}\rm}}
\def\ajoutref#1{\setbox0=\vbox{\unvbox#1\global\setbox1=
\lastbox}\unhbox1 \unskip\unskip\unpenalty}
\newif\ifpreviousitem
\global\previousitemfalse
\def\separateur{\ifpreviousitem {,\ }\fi}
\def\voidallboxes
{\setbox0=\box\auteurbox \setbox0=\box\titrebox
\setbox0=\box\titrelbox \setbox0=\box\editeurbox
\setbox0=\box\anneebox \setbox0=\box\anneelbox
\setbox0=\box\journalbox \setbox0=\box\volumebox
\setbox0=\box\pagesbox \setbox0=\box\diversbox
\setbox0=\box\collectionbox \setbox0=\null}
\def\fabriquelivre
{\ifdim\ht\auteurbox>0pt
\ajoutref\auteurbox\global\previousitemtrue\fi
\ifdim\ht\titrelbox>0pt
\separateur\ajoutref\titrelbox\global\previousitemtrue\fi
\ifdim\ht\collectionbox>0pt
\separateur\ajoutref\collectionbox\global\previousitemtrue\fi
\ifdim\ht\editeurbox>0pt
\separateur\ajoutref\editeurbox\global\previousitemtrue\fi
\ifdim\ht\anneelbox>0pt \separateur \ajoutref\anneelbox
\fi\global\previousitemfalse}
\def\fabriquearticle
{\ifdim\ht\auteurbox>0pt        \ajoutref\auteurbox
\global\previousitemtrue\fi \ifdim\ht\titrebox>0pt
\separateur\ajoutref\titrebox\global\previousitemtrue\fi
\ifdim\ht\titrelbox>0pt \separateur{\rm in}\
\ajoutref\titrelbox\global \previousitemtrue\fi
\ifdim\ht\journalbox>0pt \separateur
\ajoutref\journalbox\global\previousitemtrue\fi
\ifdim\ht\volumebox>0pt \ \ajoutref\volumebox\fi
\ifdim\ht\anneebox>0pt  \ {\rm(}\ajoutref\anneebox \rm)\fi
\ifdim\ht\pagesbox>0pt
\separateur\ajoutref\pagesbox\fi\global\previousitemfalse}
\def\fabriquedivers
{\ifdim\ht\auteurbox>0pt
\ajoutref\auteurbox\global\previousitemtrue\fi
\ifdim\ht\diversbox>0pt \separateur\ajoutref\diversbox\fi}
\def\endbibref
{\egroup \ifdim\ht\journalbox>0pt \fabriquearticle
\else\ifdim\ht\editeurbox>0pt \fabriquelivre
\else\ifdim\ht\diversbox>0pt \fabriquedivers \fi\fi\fi.\voidallboxes}

\let\styleauteur=\sc
\let\styletitre=\it
\let\styletitrelivre=\sl
\let\stylejournal=\rm
\let\stylevolume=\bf
\let\styleannee=\rm
\let\stylepages=\rm
\let\stylecollection=\rm
\let\styleediteur=\rm
\let\styleanneelivre=\rm

\begin{biblio}{99}


\begin{bibref}{Beard2}
\auteur{J. T. B. Beard Jr}  \titre{Perfect polynomials revisited}
\journal{Publ. Math. Debrecen} \volume{38/1-2} \pages 5-12 \annee
1991
\end{bibref}




\begin{bibref}{Canaday}
\auteur{E. F. Canaday} \titre{The sum of the divisors of a
polynomial} \journal{Duke Math. J.} \volume{8} \pages 721-737 \annee
1941
\end{bibref}

\begin{bibref}{Pollack}
\auteur{U. Caner Cengiz, P. Pollack and E. Trevi$\scriptsize{\tilde{N}}$o} \titre{Counting perfect polynomials} \journal{Finite Fields Appl.} \volume{47} \pages 242-255 \annee 2017
\end{bibref}





\begin{bibref}{Gall-Rahav-newcongr}
\auteur{L. H. Gallardo, O. Rahavandrainy} \titre{New congruences for odd perfect
numbers} \journal{Rocky Mountain J. Math.}
\volume{36(1)} \pages 225-235 \annee 2006
\end{bibref}

\begin{bibref}{Gall-Rahav4}
\auteur{L. H. Gallardo, O. Rahavandrainy} \titre{Odd perfect
polynomials over $\F_2$} \journal{J. Th\'eor. Nombres Bordeaux}
\volume{19} \pages 165-174 \annee 2007
\end{bibref}

\begin{bibref}{Gall-Rahav7}
\auteur{L. H. Gallardo, O. Rahavandrainy} \titre{There is no odd
perfect polynomial over $\F_2$ with four prime factors}
\journal{Port. Math. (N.S.)} \volume{66(2)} \pages 131-145 \annee
2009
\end{bibref}

\begin{bibref}{Gall-Rahav5}
\auteur{L. H. Gallardo, O. Rahavandrainy} \titre{Even perfect
polynomials over $\F_2$ with four prime factors} \journal{Intern. J.
of Pure and Applied Math.} \volume{52(2)} \pages 301-314 \annee 2009
\end{bibref}

\begin{bibref}{Gall-Rahav13}
\auteur{L. H. Gallardo, O. Rahavandrainy} \titre{Characterization of Sporadic perfect
polynomials over $\F_{2}$ } \journal{Functiones et Approx.} \volume{55.1} \pages 7-21 \annee 2016
\end{bibref}

\begin{bibref}{Gall-Rahav-Mers}
\auteur{L. H. Gallardo, O. Rahavandrainy} \titre{All even (unitary) perfect
polynomials over $\F_{2}$ with only Mersenne primes as odd divisors} \journal{Kragujevac J. Math.} \volume{49(4)} \pages 639-652 \annee 2025
\end{bibref}

\begin{bibref}{Touchard}
\auteur{J. Touchard} \titre{On prime numbers and perfect numbers} \journal{Scripta Math.} \volume{19} \pages 35-39 \annee 1953
\end{bibref}

\end{biblio}
\end{document}